\documentclass[11pt]{amsart}
\usepackage{stmaryrd}
\usepackage{textcomp}
\usepackage{enumerate}
\usepackage{setspace}
\usepackage{amssymb}
\usepackage{amsthm}
\usepackage{amsmath}
\usepackage{graphicx}
\usepackage{extarrows}
\usepackage{marvosym}
\usepackage{empheq}
\usepackage{latexsym}
\usepackage{endnotes}
\usepackage{fontenc}
\usepackage{color}
\usepackage{comment}

\allowdisplaybreaks

\newtheorem{theorem}{Theorem}

\newtheorem{lemma}[theorem]{Lemma}
\newtheorem{claim}[theorem]{Claim}
\newtheorem*{claim*}{Claim}

\theoremstyle{remark}

\def\eps{\varepsilon}

\def\E{\mathbb{E}}

\def\R{\mathbb{R}}
\def\Z{\mathbb{Z}}
\def\N{\mathbb{N}}
\def\PP{\mathbb{P}}
\def\T{\mathbb{T}}

\def\cV{\mathcal{V}}
\def\cS{\mathcal{S}}

\newcommand{\EE}{\mathbb{E}}
\renewcommand{\P}{\mathbb{P}}

\newcommand{\one}{\mathbf{1}}

\begin{document}

\title{Singularity of random symmetric matrices revisited}

\author{Marcelo Campos}
\address{Instituto Nacional de Matem\'atica Pura e Aplicada (IMPA). }
\email{marcelo.campos@impa.br}
\thanks{The first named author is partially supported by CNPq.}

\author{Matthew Jenssen}
\address{University of Birmingham, School of Mathematics.}
\email{m.jenssen@bham.ac.uk}

\author{Marcus Michelen}
\address{University of Illinois at Chicago. Department of Mathematics, Statistics and Computer Science.}
\email{michelen.math@gmail.com}

\author{Julian Sahasrabudhe}
\address{University of Cambridge. Department of Pure Mathematics and Mathematics Statistics (DPMMS).} 
\email{jdrs2@cam.ac.uk}

\begin{abstract}
Let $M_n$ be drawn uniformly from all $\pm 1$ symmetric $n \times n$ matrices. We show that the probability that $M_n$ is singular is at most
$\exp(-c(n\log n)^{1/2})$, which represents a natural barrier in recent approaches to this problem.  In addition to improving on the best-known previous bound of Campos, Mattos, Morris and Morrison of $\exp(-c n^{1/2})$ on the singularity probability, our method is different and considerably simpler.

\end{abstract}

\maketitle

\section{Introduction}
Let $A_n$ denote a random $n \times n$ matrix drawn uniformly from all matrices with $\{-1,1\}$ coefficients.  It is an old problem, of uncertain origin\footnote{See \cite{KKS} for a short discussion on the history of this conjecture}, 
to determine the probability that $A_n$ is singular. While a few moments of consideration reveals a natural \emph{lower} bound of $(1 + o(1))n^2 2^{-n+1}$, which comes from the probability that two rows or columns are equal up to sign, it is widely believed that in fact 
\begin{equation} \label{eq:sing-conj}
\P(\det A_n = 0 ) = (1 + o(1))n^2 2^{-n+1}\,.
\end{equation}
This singularity probability was first shown to tend to zero in 1967 by Koml\'os \cite{komlos}, who obtained the bound $\P(\det(A_n) = 0) = O(n^{-1/2})$.  The first exponential upper bound was established by Kahn, Koml\'{o}s, and Szemer\'{e}di~\cite{KKS} in 1995 with subsequent improvements on the exponent by Tao and Vu~\cite{TV-RSA,TV-JAMS} and Bourgain, Vu and Wood~\cite{BVW}.  In 2018, Tikhomirov \cite{Tikhomirov} settled this conjecture up to lower order terms by showing $\P(\det(A_n) = 0) = (1/2 + o(1))^n$.  Very recently, a closely related problem was resolved by Jain, Sah and Sawhney \cite{jain2020singularity}, who showed that the analogue of \eqref{eq:sing-conj} holds when the entries of $A_n$ are i.i.d.\ discrete variables of finite support that are not uniform on their support.  The conjecture \eqref{eq:sing-conj} remains open for matrices with mean-zero $\{-1,1\}$ entries.

The focus of this paper is on the analogous question for \emph{symmetric} random matrices.  
In particular, let $M_n$ denote a uniformly drawn matrix among all  $n \times n$ symmetric matrices with entries in $\{-1,1\}$.  
In this setting it is also widely believed that $\P(\det M_n = 0) = \Theta(n^2 2^{-n})$ as in the asymmetric case \cite{CMMM,costello-tao-vu,vu-randomdiscretematrices} although here much less is known.  For instance, the fact that $\P(\det M_n = 0) = o(1)$, was only resolved in 2005 by Costello, Tao and Vu \cite{costello-tao-vu}. Subsequent superpolynomial upper bounds of the form $n^{-C}$ for all $C$ and $\exp(-n^{c})$ were proven respectively by Nguyen  \cite{nguyen-singularity} and Vershynin \cite{vershynin-invertibility} by different techniques: Nguyen used an inverse Littlewood-Offord theorem for quadratic forms based on previous work by Nguyen and Vu \cite{nguyen-vu-1,nguyen-vu-2}, while Vershynin used a more geometric approach pioneered by Rudelson and Vershynin \cite{RV3,RV1,RV2}.

A combinatorial approach developed by Ferber, Jain, Luh and Samotij \cite{ferber2019counting} was applied by Ferber and Jain \cite{ferber-jain} in 2018 to prove that 
$\P(\det M_n = 0) \leq \exp\big(-c n^{1/4}(\log n)^{1/2}\big)$.  
Another combinatorial approach was taken by Campos, Mattos, Morris and Morrison \cite{CMMM} who achieved the bound $\P(\det M_n = 0) \leq \exp\big(-c n^{1/2}\big)$.  Their argument centers around an inverse Littlewood-Offord theorem inspired by the method of hypergraph containers.

The proofs of \cite{CMMM,ferber-jain,vershynin-invertibility} all follow the same general shape: divide all potential vectors $v$ for which we could have $M_nv = 0$ into ``structured'' and ``unstructured'' vectors, show that the unstructured vectors do not contribute, and union bound over the structured vectors. The main difficulty (and novelty) in these proofs arises in a careful understanding of the contribution of the \emph{structured} vectors.

While we have this method to thank for the recent successes on this problem, an important limitation was pointed out in \cite[Section 2.2]{CMMM} who argued that this method could not provide any improvement to the singularity probability beyond $\exp(-c\sqrt{n\log n})$, provided the randomness in the matrix is not ``reused''. Here we show that this natural ``barrier'' is attainable.

\begin{theorem}\label{thm:main}
Let $M_n$ be drawn uniformly from all $n \times n$ symmetric matrices with entries in $\{-1, 1\}$. Then for $c = 2^{-13}$ and $n$ sufficiently large
\[ \PP( \det(M_n) = 0 ) \leq \exp\left( -c\sqrt{n\log n}\right).\]
\end{theorem}

Indeed, our proof of Theorem~\ref{thm:main} follows the shape of \cite{CMMM,ferber-jain,vershynin-invertibility} and improves upon these results primarily by proving 
an improved and considerably simpler ``rough'' inverse Littlewood-Offord theorem. This theorem parallels Theorem 2.1 in \cite{CMMM}. 

To state this, we need a few notions. For a vector $v\in \Z_p^n$ and $\mu\in [0,1]$,
we define the random variable $X_\mu(v) := \eps_1v_1+\cdots+\eps_n v_n$, where $\eps_i \in \{-1,0,1\}$ are i.i.d.\ and $\PP(\eps_i = 1) = \PP(\eps_i = - 1) = \mu/2$.
Also define $\rho_{\mu}(v) = \max_x \P( X_{\mu}(v) = x)$ and\footnote{We will also write $\rho_1(v) = \rho(v)$.} let $|v|$ denote the number of non-zero entries of $v$.
Finally for $T\subseteq [n]$, let $v_T  := (v_i)_{i \in T}$.

We now introduce a simple concept that is key to our rough inverse Littlewood-Offord theorem. For a vector $w = (w_1,\ldots,w_d)$ we define the
\emph{neighbourhood} of $w$ (relative to $\mu$) as 
\begin{align}\label{eq:Ndef}
N_{\mu}(w)  := \{ x \in \Z_p : \PP(X_\mu(w) = x) > 2^{-1}\PP(X_\mu(w) = 0) \}, \end{align}
which is the set of places where our random walk is ``likely'' to terminate, relative to $0$.

\begin{theorem}\label{thm:rough-ILwO}
Let $\mu \in (0,1/4]$, $k, n \in \mathbb{N}$, $p$ prime and $v \in \Z_p^n$.
Set $d=\frac{2}{\mu}\log \rho_\mu(v)^{-1}$, suppose that $|v|\geq kd$ and $\rho_\mu(v)\geq \frac{2}{p}$.
Then there exists $T\subseteq [n]$ with $|T|\leq d$ so that if we set $w=v_T$
then
$v_i \in N_{\mu}(w)$ for all but at most $kd$ values of $i\in [n]$ and 
\[
|N_\mu(w)|\leq \frac{256}{k^{1/5}} \cdot \frac{1}{\rho_{\mu}(v)}\, .
\]
\end{theorem}

In addition to controlling the number of such vectors $v$ with prescribed $\rho(v)$, 
Theorem~\ref{thm:rough-ILwO} gives some further information on the structure of 
the sets that the $v_i$ belong to, which makes for a simplified application of Theorem~\ref{thm:rough-ILwO} in the proof of Theorem~\ref{thm:main}. 

The proof of Theorem~\ref{thm:rough-ILwO} has two parts. The first can be found in Section~\ref{sec:fourier} and uses Fourier analysis in the style of Hal\'asz \cite{halasz}, whose influential techniques pervade the literature. The second is a novel (and simple) iterative application
of a greedy algorithm. This can be found in Section~\ref{sec:proof} along with the proof of Theorem~\ref{thm:rough-ILwO}.

In what follows we discuss the proof of Theorem~\ref{thm:main}. In addition to illustrating the method of \cite{ferber-jain, CMMM} in a little more detail, we hope the reader will get some feeling for why Theorem~\ref{thm:rough-ILwO} is so integral to the problem.

\subsection{Discussion of proof}

The event `$M_n$ is singular' can, somewhat daftly, be expressed as  $\bigcup_{v \in \R^n \setminus \{0\}} \{ Mv = 0  \}$.
To reduce the size of this unwieldy union, we notice that it is sufficient to consider all non-zero $v \in \Z^n$  and then reduce modulo $p$, for a prime $p \approx \exp\big(c(n\log n)^{1/2}\big)$. Since the probability that $Mv$ is zero is certainly bounded by the probability $Mv$ is zero modulo $p$, it is enough for us
to upper bound the probability of the event $\bigcup_{v \in \Z_p^n \setminus \{0\}} \{ Mv = 0  \}$, where all operations are taken over the field of $p$ elements.

Having reduced our event to a union of a finite number of sets, it is temping to greedily apply the union bound to the events $\{ Mv = 0  \}$, for non-zero $v \in \Z_p^n$. Unfortunately in our case, a small wrinkle arises with vectors for which $\rho(v) \approx 1/p$;
that is, very close to the ``mixing'' threshold. To get around this, we again follow \cite{ferber-jain, CMMM} and use a lemma that allows us to safely exclude all $v$ with
$\rho(v) < cn/p$ from our union bound, at the cost of working with a slightly different event which, in practice, adds little difficulty to our task.

\begin{lemma}\label{lem:redux}
Let $c=1/800,\, n\in \N$ sufficiently large and
$p\leq \exp(c\sqrt{n\log n})$ be a prime.  If for all $\beta = \Theta(n/p)$ we have
\begin{equation} \label{eq:union} \sum_{v : \rho(v) \geq \beta} \max_{w \in \Z_p^n} \PP(Mv = w)  \leq e^{-cn}\end{equation}
then for $c' = 2^{-13}$
\[ \PP\big( \det(M_n) = 0 \big) \leq \exp(-c' \sqrt{n\log n}).\]
\end{lemma}

This is essentially Lemma 2.1 in \cite{CMMM}, but is implicit in the earlier work of \cite{ferber-jain} who proved this using the earlier ideas of \cite{costello-tao-vu}. We provide the short formal derivation of Lemma~\ref{lem:redux}
from Lemma 2.1 of \cite{CMMM} in Section~\ref{sec:proof-of-main}.

With Lemma~\ref{lem:redux} in hand, our task is now clear:  we need to bound the sum on the left hand side of \eqref{eq:union}. 
To do this, we invoke our inverse Littlewood-Offord result (Theorem~\ref{thm:rough-ILwO}, in the form of Lemmas \ref{lem:initialcont-lem} and \ref{lem:container-lem}) and prove the following. 

\begin{theorem}\label{thm:qnbeta}
Let $c=1/800,\,  n\in \N$ sufficiently large and
$p\leq \exp(c\sqrt{n\log n})$ prime.  Then for  $\beta = \Theta(n/p)$ we have
$$ \sum_{v : \rho(v) \geq \beta} \max_{w \in \Z_p^n} \PP(Mv = w)    \leq e^{-cn}\,.$$
\end{theorem}

\noindent\textbf{Remark:} Simultaneously to our work, Jain, Sah and Sawhney~\cite{jainnew} obtained an upper bound on the singularity probability of the form $\exp(-c n^{1/2}(\log n)^{1/4})$ and a bound on the lower tail of
the least singular value for symmetric random matrices with subgaussian entries.

\section{Proof of Theorem~\ref{thm:rough-ILwO}}\label{sec:proof}
In this section we prove Theorem~\ref{thm:rough-ILwO} modulo a key Fourier lemma which we postpone to Section~\ref{sec:fourier}.
To go further, we introduce a little notation. Let $\Z_p^{\ast}$ denote the set of all vectors of finite dimension with  entries in $\Z_p$. For $v = (v_1,\ldots, v_k) ,w = (w_1,\ldots, w_l) \in \Z^{\ast}_p$, let $vw :=(v_1,\ldots, v_k,w_1,\ldots, w_l)$ denote the \emph{concatenation} of $v$ and $w$ and let $v^k$ denote the concatenation of $k$ copies of $v$. 
For $v\in \Z_p^n$ and $T\subseteq [n]$, let $v_T  := (v_i)_{i \in T}$ and say that $w$ is a \emph{subvector} of $v$ if $w=v_T$ for some $T\subseteq [n]$.  
We also define $|v|$ to the be size of the \emph{support} of $v$,  the number of non-zero coordinates.

Unless specified otherwise, take $\mu=1/4$ for definiteness. We recall the key definition introduced in \eqref{eq:Ndef}.
For $w\in \Z_p^{\ast}$, we define the \emph{neighbourhood} of $w$ as
\[ N(w)  := \{ x \in \Z_p : \PP(X_\mu(w) = x) > 2^{-1}\PP(X_\mu(w) = 0) \}. \]

This is motivated by the fact that for $\mu \in [0,1/2]$, the walk $X_{\mu}$ is most likely to be found at $0$ (see e.g.\  \cite[Corollary 7.12]{tao-vu-book}),
\begin{equation} \label{eq:lazy-loves-home} \rho_\mu(w)= \PP(X_\mu(w)=0) \,  .\end{equation}
Hence, we may think of $N(w)$ as the set of all values of the random walk $X_\mu(w)$, which are at least half as likely as the most likely value. 
We can also easily control the size of $N(w)$. Indeed,
\[
1\geq \sum_{x\in N(w)} \PP(X_\mu(w)= x)> \frac{1}{2}|N(w)|\cdot\PP(X_\mu(w) = 0)=\frac{1}{2}|N(w)|\rho_\mu(w)
\]
and so 
\begin{equation}\label{eq:Lw}
|N(w)|\leq \frac{2}{\rho_\mu(w)}\, .
\end{equation}

We now turn to our greedy algorithm which, given a vector $v\in \Z_p^\ast$, returns a short subvector
$w$ of $v$ such that each coordinate of $v$ is contained in $N(w)$.
The following simple lemma can be interpreted as an inverse Littlewood-Offord result in its own right,
and is \emph{almost} as good as Theorem~\ref{thm:rough-ILwO}, however it only gives a bound of $|N(w)| \leq 1/\rho_{\mu}(w) \leq 1/\rho_{\mu}(v)$, which is
lacking the crucial factor of $k^{-1/5}$. For this lemma we use the monotonicity of $\rho_{\mu}$ \cite[Corollary 7.12]{tao-vu-book}: if 
$w, v\in \Z^{\ast}_p$ where $w$ is a subvector of $v$, then
\begin{equation} \label{eq:monotone} \rho(v)\leq \rho_\mu(w) \, .\end{equation}

\begin{lemma}\label{lem:initialcont-lem}
For $\mu\in (0,1/4]$ and $n\in \N$, let $v\in \Z_p^n$. Then there exists $T\subseteq [n]$,  
such that $v_i\in N(v_T)$ for all $i\notin T$, $\rho_\mu(v_T)\leq (1-\mu/2)^{|T|}$ and so 
\[|T|\leq \frac{2}{\mu}\log \frac{1}{\rho_{\mu}(v)}.\]
\end{lemma}
\begin{proof}
We build a sequence of sets $T_1, \ldots, T_d\subseteq [n]$ with $|T_i|=i$ via the following greedy process. Let $T_1=\{1\}$.
Given $T_t\subseteq [n]$ with $|T_t|=t$ for $t\geq 1$, 
let $v_{T_t}=(x_1,\ldots, x_t)$.
Pick $i\in [n]\backslash T_t$
such that 
\begin{align}\label{eq:stept}
\rho_\mu(x_1\ldots x_{t}v_i)\leq (1-\mu/2) \rho_\mu(x_1\ldots x_{t})\, .
\end{align}
If no such $i$ exists we terminate the process and set $T=T_t$.
Suppose this process runs for $d$ steps producing 
$T\subseteq [n]$ such that $v_T=(x_1, \ldots, x_d)$. 
By the termination condition, 
we have that for $i\in [n]\backslash T$
\[
\rho_\mu(x_1\ldots x_d v_i)> (1-\mu/2) \rho_\mu(x_1\ldots x_{d})\, .
\]
Conditioning on the coefficient of $v_i$ and 
using that $\PP(X_\mu(x_1\ldots x_d )= v_i)=\PP(X_\mu(x_1\ldots x_d) = -v_i)$
by symmetry, we can rewrite the left hand side to obtain
\[
\mu \PP(X_\mu(x_1\ldots x_d) = v_i) + (1-\mu)\rho_\mu(x_1\ldots x_d) > (1-\mu/2) \rho_\mu(x_1\ldots x_{d})\, .
\]
Rearranging shows $v_i \in N(v_T)$. 
For the bound on $d = |T|$, observe that by \eqref{eq:monotone},
inequality \eqref{eq:stept} and the fact that $\rho_\mu(x_1)=(1-\mu)$
we have
\[
\rho_\mu(v)\leq \rho_\mu(x_1\ldots x_d)\leq (1-\mu/2)^d\leq e^{-\mu d/2}\, .
\]
\end{proof}

The next lemma shows that we can improve Lemma~\ref{lem:initialcont-lem} by applying it iteratively. This will be key to regaining this crucial $k^{-1/5}$
in Theorem~\ref{thm:rough-ILwO}, and will ultimately give our $\sqrt{\log n}$ gain in the exponent of the singularity probability. In fact, if one only wanted to prove a bound of the form $\exp(-cn^{1/2})$ using our method, one needs only to use Lemma~\ref{lem:container-lem} along with a simplified Fourier argument. 

For this lemma we need the following property of $\rho_{\mu}$, which can be found in \cite[Corollary 7.12]{tao-vu-book}. Let $w_1,\ldots,w_k\in \Z_p^{\ast}$ and $\mu \in (0,1/2)$ then \begin{equation}\label{eq:holder}
\rho_\mu(w_1\cdots w_k) \leq \max_{ j \in [k] }\, \rho_\mu\left(w_j^k\right)\, .\end{equation}

\begin{lemma}\label{lem:container-lem}
Let $\mu\in (0,1/4]$, $n\in \N$
and $v \in \Z_p^n$. Set $d= \frac{2}{\mu}\log \rho_\mu(v)^{-1}$
and let $k \in \N$ be such that $kd\leq n$.
Then there exists $T\subseteq S\subseteq [n]$ with $|T|\leq d$, $|S| \leq kd$ such that $v_{i}\in N(v_T)$ for all $i\not \in S$ and $\rho_\mu(v_S)  \leq \rho_\mu( v_T^k )$. \end{lemma}
\begin{proof}
We will define a sequence of sets $[n]=A_1\supseteq \cdots\supseteq A_k\supseteq A_{k+1}$. Given $v_{A_j}$, we choose $T_j\subseteq [n]$ with $ v_{T_j}= (x_1,\ldots,x_{d(j)})$ given by Lemma \ref{lem:initialcont-lem} applied to $v_{A_j}$ and let

 \[A_{j+1}=A_j\setminus T_j \qquad\text{and} \qquad S=\bigcup_{j=1}^k T_j.\] 
By Lemma~\ref{lem:initialcont-lem}, we have that $v_{i}\in N(v_{T_{j}})$ for all $i\in A_{j+1}$.
In particular, since $S^c\subseteq A_j$ for all $1\leq j\leq k+1$, $v_{i}\in N(v_{T_{j}})$ for all $i\not\in S$ and $1\leq j\leq k$. Note also that $|T_j|\leq d$ for all $1\leq j\leq k$. 

 Let $T$ be the $T_j$ for which $\rho_\mu(v_{T_j}^k)$ is maximized.
 The first claim of the lemma follows from the above. 
 For the second claim note that, 
 by \eqref{eq:holder} we have 
\[\rho_\mu(v_S)\leq \max_{1\leq j\leq k}\rho_\mu(v_{T_j}^k)=\rho_\mu(v_{T}^k)\, .\]
\end{proof}

To conclude the proof of our Theorem~\ref{thm:rough-ILwO}---and to understand the strength of Lemma \ref{lem:container-lem}---we introduce our main Fourier ingredient, the proof of which is found in Section~\ref{sec:fourier}.

\begin{lemma}\label{cor:relwin}
Let $\mu \in (0,1/4]$, $k\in \N$ and $v \in \Z_p^\ast$ such that $|v|\neq 0$. Then
\[ \rho_\mu(v^k) \leq  64 k^{-1/5}\rho_\mu(v) + p^{-1}\,. \]
\end{lemma}

\begin{proof}[Proof of Theorem~\ref{thm:rough-ILwO}]
Let $k, n \in \mathbb{N}$ and $v \in \Z_p^n$ be as in the theorem statement. 
By Lemma~\ref{lem:container-lem}, 
there exists $T\subseteq S \subseteq [n]$ with $|T|\leq d$, $|S| \leq kd$ such that $v_{i}\in N(v_T)$ for all $i\not \in S$ and $\rho_\mu(v_S)  \leq \rho_\mu( v_T^k )$. Moreover, since $|v|\geq kd$, the support of $v_T$ is non-zero. Applying Lemma~\ref{cor:relwin} we conclude that 
\[ \rho_\mu(v_S)\leq \rho_\mu(v_T^k) \leq  64 k^{-1/5}\rho_\mu(v_T) + p^{-1} \,. \]
By~\eqref{eq:Lw} and~\eqref{eq:monotone} we then have
\[
|N_\mu(v_T)|\leq \frac{2}{\rho_\mu(v_T)}\leq  \frac{128}{ k^{1/5 }(\rho_\mu(v_S)-p^{-1})}\leq  \frac{256}{ k^{1/5 }\rho_\mu(v)}\, ,
\]
where on the final bound  we use that $\rho_\mu(v)\geq \frac{2}{p}$.
\end{proof}

\section{Proof of Theorem~\ref{thm:main}}\label{sec:proof-of-main}

In this section we prove Theorem~\ref{thm:qnbeta} which, from the discussion in the introduction, implies Theorem~\ref{thm:main}. As we were a bit quick
with this discussion, we take a moment to spell out the proof of this implication. 

Define 
\[q_n(\beta):=\max_{w\in \Z_p^n}\PP\left(\exists~v\in \Z_p^n\setminus\{0\}:~M\cdot v=w~\text{and}~\rho(v)\geq \beta\right)\]
and note the following lemma from \cite{CMMM} (their Lemma 2.1).
\begin{lemma}\label{lem:CMMM}
Let $n\in \N$ and $p>2$ be a prime. Then for every $\beta>0$ \[\PP(\det(M_n)=0)\leq n\sum_{m=n-1}^{2n-3}\left(\beta^{1/8}+\frac{q_m(\beta)}{\beta}\right)\,.\]
\end{lemma}
Our Lemma~\ref{lem:redux} follows easily.

\begin{proof}[Proof of Lemma \ref{lem:redux}]
	Pick a prime $p = t \exp(c \sqrt{n \log n})$ with $c = 1/800$ and $t\in[1/2,1]$.  Looking to apply Lemma \ref{lem:CMMM} with $\beta = \Theta(n/p)$, we apply the union bound to $q_n(\beta)$ and our assumption for each $n-1 \leq m \leq 2n-3$ to bound 
\[ q_n(\beta) \leq 	\sum_{v : \rho(v) \geq \beta} \max_{w \in \Z_p^n} \PP(Mv = w)    \leq e^{-cn}. \]
	Thus, we apply Lemma \ref{lem:CMMM}, to obtain
	$$\P(\det(M_n) = 0) \leq e^{-c(1 + o(1)) \sqrt{n \log n}/8 } + e^{-cn(1+o(1))} \leq e^{-c\sqrt{ n \log n}/9}\,, $$
for $n$ sufficiently large.	
\end{proof}

\vspace{3mm}

With these reductions firmly in-hand, we turn to prove Theorem~\ref{thm:qnbeta}, and therefore Theorem~\ref{thm:main}.

\begin{proof}[Proof of Theorem~\ref{thm:qnbeta}]
Throughout we assume that $n$ is sufficiently large so that all inequalities in the proof hold, we let
 $k=n^{1/4}$, $d=\frac{2}{\mu}\log p\leq \frac{2}{\mu}\sqrt{n\log n}$ and 
define $\mathcal{V}:=\{v\in \Z_p^n\backslash \{0\} : \rho_\mu(v)\geq \beta\}$. Our task is to bound 
\begin{equation} \label{eq:the-sum} Q_n(\beta) := \sum_{v \in \mathcal V} \max_w\, \P(M_n\cdot v = w). \end{equation}
We start our analysis of \eqref{eq:the-sum} by partitioning this sum by way of a function $f : \cV \rightarrow \cS$. To define $f$, let $v\in \Z_p^n$ and apply 
Lemma \ref{lem:container-lem} to obtain $S,T \subseteq [n]$. We then apply Lemma~\ref{lem:initialcont-lem} to $v_S$ to obtain a further set $T'\subseteq [n]$. 
We then define $f(v)=(S,T,T',v_T,v_{T'})$ and put $\cS := f(\cV)$. We thus partition our sum \eqref{eq:the-sum} as
\begin{equation} \label{eq:the-sum-2} Q_n(\beta) = \sum_{s\in \mathcal{S}}\sum_{v\in f^{-1}(s)} \max_w\,  \P(M_n\cdot v = w). \end{equation}

 Note that if $s=(S,T,T',w_1,w_2)\in \mathcal{S}$, then
 \begin{equation}\label{eq:sprop}
 |S|\leq kd,\quad |w_1|, |w_2| \leq d, \quad \rho_\mu(w_1)\geq \beta, \quad \rho_\mu(w_2)\leq (1-\mu/2)^{|w_2|} \text{\,  and \, } w_2\neq 0\, 
 \end{equation} by Lemmas \ref{lem:initialcont-lem} and~\ref{lem:container-lem} together with \eqref{eq:monotone}, and note that we have the bound 
\begin{equation}\label{eq:Sbd} |\cS|\leq 8^n p^{2d}\, ,\end{equation}
  since there are $8^n$ choices for $S,T,T'$ and at most $p^{2d}$ choices for $w_1,w_2$.

We now turn to bounding a given term in the sum \eqref{eq:the-sum-2}, based on which piece of the partition it is in.
Let $s = (S,T,T',w_1,w_2) \in \cS$ and $v\in f^{-1}(s)$. For any $w\in \Z_p^n$, we bound $\P(M_n\cdot v = w)$ by 
first revealing the rows indexed by $S^c$ and then revealing the 
rows indexed by $S\backslash T'$,
\[ \PP(M\cdot v=w) \leq  \PP\left(M_{(S\backslash T')\times [n]} \cdot v= w_{S\backslash T'} \mid M_{S^c\times[n]}\cdot v = w_{S^c} \right) 
\cdot \PP(M_{S^c\times[n]}\cdot v = w_{S^c}). \]
Looking only on the off-diagonal blocks $(S \setminus T') \times T'$ and $S^c \times S$ and considering the ``worst case'' vectors for these blocks, we have
\[ \PP(M\cdot v=w)  \leq \max_{u}\, \PP(M_{(S\backslash T') \times T'}\cdot v_{T'} = u) \cdot \max_{u}\, \PP(M_{S^c \times S} \cdot v_S = u). \]
The crucial point here is that these events can be written as an intersection of independent events concerning the rows. That is
\begin{equation}\label{eq:P-bound} \PP(M\cdot v=w) \leq \rho(v_{T'})^{|S|-|T'|} \rho(v_S)^{n-|S|} \leq \rho_\mu(v_{T'})^{|S|-|T'|} \rho_\mu(v_S)^{n-|S|},  \end{equation}
where this last inequality follows from the monotonicity of $\rho$ in the parameter $\mu$, noted at \eqref{eq:monotone}.

We now bound the size of a piece of our partition $|f^{-1}(s)|$.
By~\eqref{eq:Lw} together with Lemmas~\ref{lem:initialcont-lem} and~\ref{lem:container-lem}, the number of choices for $v_{S^c}$ and $v_{S\setminus T'}$  are (respectively) at most
\[ |N(w_1)|^{n-|S|}\leq\left(\frac{2}{\rho_\mu(w_1)}\right)^{n-|S|}, \qquad |N(w_2)|^{|S|-|T'|}\leq\left(\frac{2}{\rho_\mu(w_2)}\right)^{|S|-|T'|}, \] 
so that 
\begin{align}\label{eq:fsbound}
|f^{-1}(s)|\leq  \left(\frac{2}{\rho_\mu(w_1)}\right)^{n-|S|}\left(\frac{2}{\rho_\mu(w_2)}\right)^{|S|-|T'|}\, .
\end{align}
By~\eqref{eq:P-bound} and the fact that $|S|\leq kd=o(n)$ (by our choice of parameters), we have
 \begin{equation}\label{eq:finv-w1}
 \sum_{v\in f^{-1}(s)}\max_w\PP(M_{n}\cdot v=w)\leq 2^n
 \left(\frac{\rho_\mu(w_1^k)}{\rho_\mu(w_1)}\right)^{n-|S|}
 \leq
  2^n\left(\frac{\rho_\mu(w_1^k)}{\rho_\mu(w_1)}\right)^{24n/25}
 \, .\end{equation}
We consider first the case where $|w_1| \neq 0$; then 
we may apply Lemma \ref{cor:relwin} to obtain the bound
\[
\rho_\mu(w_1^k)\leq 64(\mu k)^{-1/5} \rho_\mu(w_1)+\frac{1}{p}\, .
\] 
By the bound $\rho_\mu(w_1)\geq \beta =  \Theta(n/p)$, we then have
\[
\frac{\rho_\mu(w_1^k)}{\rho_\mu(w_1)}\leq 64(\mu k)^{-1/5} +\Theta(n^{-1}) \leq n^{-1/24}\, .
\]

Combining this with \eqref{eq:Sbd} and \eqref{eq:finv-w1} 
shows that
\begin{align}
\nonumber\sum_{\substack{s\in \mathcal{S},\\ |w_1|\neq 0}}\sum_{v\in f^{-1}(s)}\max_w\PP(M_{n}\cdot v=w) &\leq |\mathcal{S}| \cdot n^{-n/25} \leq 8^n p^{2d}n^{-n/25}  \\
\label{eq:sum1}&\leq 8^n \exp\left(\frac{4c}{\mu}n \log n- \frac{1}{25} n \log n\right) \leq e^{-n}\, ,
\end{align}
provided $c\leq\mu/200$.
Now if $|w_1|=0$ then there are at most 
\[
|f^{-1}(s)|\leq \left(\frac{2}{\rho_\mu(w_2)}\right)^{|S|-|T'|}
\]
choices for $v$.  Notice that $\rho_\mu(v_S)\leq \rho_\mu(w_2)$ and so
\[\sum_{v\in f^{-1}(s)}\max_w\PP(M_{n}\cdot v=w)\leq \rho_{\mu}(w_2)^{n-|T'|} \left(\frac{1}{\rho_{\mu}(w_2)}\right)^{|S|-|T'|}\leq \rho_{\mu}(w_2)^{n/2}\leq (1-\mu/2)^{n|w_2|/2}\, ,\]
where for the final inequality we used~\eqref{eq:sprop}.
On the other hand, by~\eqref{eq:sprop}, the number of choices for $s=(S, T, T', w_1, w_2)$ such that $|w_1|=0, |w_2|=t$ is at most 
\[
\binom{n}{\leq kd}^3 p^{t}\leq \exp(ct\cdot\sqrt{n\log n}+3kd\log n).
\]
Putting our bounds together, we have
\[
\sum_{\substack{s\in \mathcal{S},\\ |w_1|=0, |w_2|=t}}\sum_{v\in f^{-1}(s)}\max_w\PP(M_{n}\cdot v=w)\leq \exp(ct\cdot\sqrt{n\log n}+3kd\log n-n\mu t/4)\leq e^{-n\mu t/5}\, .
\]
Summing over all $t\geq 1$ (recalling that $w_2\neq 0$) and using~\eqref{eq:sum1}, we conclude that
\[
Q_n(\beta)= \sum_{s\in \mathcal{S}}\sum_{v\in f^{-1}(s)}\max_w\PP(M_{n}\cdot v=w) \leq e^{-\mu n/6}\, ,
\] as desired.\end{proof}

\section{Proof of Lemma~\ref{cor:relwin}}\label{sec:fourier}
In this section, we pin down one final loose end, the proof of Lemma~\ref{cor:relwin}, which is our main Fourier lemma.
For $v \in \Z_p^n$, and $\mu\in[0,1]$ we note a standard Fourier expression for $\rho_\mu(v)$.   Define 
\begin{equation}\label{eq:transform}  f_{\mu,v}(\xi) := \prod_{i = 1}^n ((1 - \mu) + \mu c_p(v_i \xi)), \end{equation}
where we let $c_p(x) = \cos(2\pi x/p)$. We then have
\begin{equation} \label{eq:halasz} \rho_{\mu}(v) = \E_{\xi \in \Z_p} f_{\mu,v}(\xi)\,.\end{equation}
Clearly $|f_{\mu,v}(\xi)| \leq 1$ and for $\mu \leq 1/2$ each of the terms in the product $f_{\mu,v}(\xi)$ is non-negative.
In this case it is natural 
to work with $\log f_{\mu,v}$.
For this, we let $\| x \|_{\T}$ denote the distance from $x\in \R$ to the nearest integer and note the following bounds.
For $\mu\in [0,1/4]$ we have 
\begin{equation}\label{eq:cos-approx}  \mu \| x/p \|_{\T}^2 \leq - \log \left(1 - \mu + \mu c_p( x) \right) \leq 32\mu \|x/p \|_{\T}^2\, , \end{equation}	
which are elementary\footnote{For these explicit constants, note the bounds $a \leq -\log(1 - a) \leq (3/2)a$ for $a \in [0,1/4]$ and $x^2 \leq 1 - \cos(2\pi x) \leq 20 x^2$ for $|x| \leq 1/2$.} and can be found in (7.1) in \cite{tao-vu-book}.

For the following lemma, one of the main results of this section, we need the well-known Cauchy-Davenport inequality which tells us that for $A,B \subseteq \Z_p$
we have $|A+B| \geq \min\{ |A|+|B|-1, p\}$. Here, as usual, $A+B := \{ a+ b : a\in A, b \in B \}$.

A first step towards Lemma \ref{cor:relwin} is to prove it in the case when $\rho_\mu(v)$ is not too large.

\begin{lemma}\label{lem:relwin} 
	Let $\mu \in (0,1/4]$, $v \in \Z_p^\ast$ and $k\in \N$. Then
	\[ \rho_\mu(v^k) \leq \Big( \rho_\mu(v)^{\frac{k-1}{k}} + \frac{8}{\sqrt{\mu k}} \Big) \rho_\mu(v) + p^{-1} .\]
\end{lemma}

To prove this lemma, we adopt some temporary notation. Let $F = f_{\mu,v^k}$ and $G = f_{\mu,v}$, be as defined in \eqref{eq:transform} and note that $G = F^{1/k}$. We note also that $F$ is non-negative since $\mu \leq 1/4$. Let $\ell := \tfrac{1}{8}(\mu k)^{1/2} $. For all $\alpha \in (0,1)$, we consider the level sets 
\[ A_{\alpha} := \{ \xi \in \Z_p : F(\xi) > \alpha \} \qquad  B_{\alpha} := \{\xi \in \Z_p : G(\xi) > \alpha \}.\] 

\begin{claim}\label{claim:CD}
For $\alpha\in(0,1)$, we have $\ell \cdot A_{\alpha} \subseteq B_{\alpha}$. \end{claim}
\begin{proof}
	To see this, assume $\xi_1,\ldots,\xi_{\ell} \in A_{\alpha}$ and so $G(\xi_i) = (F(\xi_i))^{1/k} > \alpha^{1/k}$ for each $i \in [\ell]$. 
	Taking logs of both sides and applying \eqref{eq:cos-approx} gives, for each $i \in [\ell]$,
	\begin{equation}\label{eq:logbound} \mu \sum_{j = 1}^n \| \xi_i v_j \|_{\T}^2 \leq  -\log G(\xi_i)  \leq k^{-1}\log \alpha^{-1}\,.  \end{equation}
	Thus, using the triangle inequality along with \eqref{eq:logbound} gives
	$$ \left( \sum_{j = 1}^n \|(\xi_1 + \cdots + \xi_{\ell}) v_j \|_{\T}^2\right)^{1/2} \leq 
	\sum_{i = 1}^{\ell} \left(\sum_{j = 1}^n \| \xi_i v_j\|_{\T}^2 \right)^{1/2} 
	\leq  \ell \left(\frac{\log \alpha^{-1}}{\mu k}\right)^{1/2} \,.$$
	It then follows from the upper bound in \eqref{eq:cos-approx} that
	\[ - \log G(\xi_1 + \cdots  + \xi_{\ell} ) \leq 32 \sum_{j = 1}^n \|(\xi_1 + \cdots + \xi_{\ell}) v_j \|_{\T}^2 \leq 32 \frac{\ell^2}{\mu k} \log \alpha^{-1}.  \] 
	Thus, using our choice of $\ell = \frac{1}{8}(\mu k)^{1/2}$, we have $G(\xi_1 + \cdots  + \xi_{\ell} ) > \alpha$,  and so $\xi_1 + \cdots  +\xi_{\ell} \in B_{\alpha}$. \end{proof}

\begin{proof}[Proof of Lemma~\ref{lem:relwin}]
Letting $g:=\EE_\xi G=\rho_\mu(v)$,
	we want to show that $\EE_{\xi} F \leq \left(g^{(k-1)/k}+\frac{8}{\sqrt{\mu k}}\right)g  + p^{-1}$. We do this in two ranges. First we recall that $F^{1/k}= G$ and so
	\[ \EE_\xi \left[ F \one(F\leq g) \right]\leq  \EE_\xi \left[G\cdot g^{(k-1)/k}\right] = g^{\frac{2k-1}{k}}. \]
	Next we treat the $\xi$ for which $F(\xi) > g$. First note that by Markov's inequality $|B_{\alpha}| < p$, for all $\alpha>g$. It follows from Claim~\ref{claim:CD} and the Cauchy-Davenport inequality that $|A_\alpha| \leq \ell^{-1}|B_\alpha|+1$ for all $\alpha>g$. Thus,
	\[ \EE_\xi\left[ F \one\left( F > g \right)\right] =  \int_{g}^{1} |A_t|p^{-1} \, dt \leq  \ell^{-1}\int_{g}^{1} |B_t|p^{-1} \, dt + 1/p \leq g/\ell + 1/p.\]
	 Putting our bounds together we have
\[ \rho_{\mu}(v^k) \leq  \left(g^{(k-1)/k} + \frac{8}{\sqrt{\mu k}}\right)g + 1/p = \left(\rho_\mu(v)^{(k-1)/k}  + \frac{8}{\sqrt{\mu k}}\right)\rho_{\mu}(v) + p^{-1},\]
	as desired.
\end{proof}

To complete our proof of Lemma~\ref{cor:relwin}, we need the following classical result: 
\begin{lemma}
	If $v\in \Z_p^\ast$ with $v\neq0$ then $\rho_{\mu}(v) = \frac{64}{\sqrt{\mu |v|}} + p^{-1}$.
\end{lemma}
Letting $d=|v|$, this lemma may be deduced by bounding $\rho_{\mu}(v) \leq \rho_\mu(v_j^d)$ for some $j$ by ~\eqref{eq:holder}, noting that $\rho_\mu(v_j^d) = \rho_\mu(1^d)$ and bounding the latter either directly or using a standard local central limit theorem.  Alternatively, a stronger statement may be found in \cite[Lemma 2.3]{CMMM}.

\begin{proof}[Proof of Lemma~\ref{cor:relwin}]
	If $\rho_\mu(v)\leq (\mu k)^{-1/4}$ then
	Lemma~\ref{lem:relwin} tells us that $\rho_{\mu}(v^k) \leq 64(\mu k)^{-1/5}\rho_{\mu}(v) + p^{-1}$, as desired. On the other hand, if $\rho_\mu(v)> (\mu k)^{-1/4}$, 
	\begin{align*}
	\rho_{\mu}(v^k) = \frac{64}{\sqrt{\mu k |v|}}  + 1/p \leq  64 (\mu k)^{-1/4} \rho_{\mu}(v) + 1/p 
	\end{align*} thus completing the proof. \end{proof}

\bibliographystyle{abbrv}
\bibliography{Bib}

\end{document}